\documentclass{amsart}

\usepackage{amssymb}
\usepackage{amscd}
\usepackage{hyperref}

\newtheorem*{GHthm}{Grove-Halperin Theorem}
\newtheorem*{MAINthm}{Main Theorem}
\newtheorem{thm}{Theorem}

\newtheorem{lem}[thm]{Lemma}

\begin{document}

\title{Rational Ellipticity in Cohomogeneity Two}
\date{\today}
\author{Joseph E. Yeager}
\address{Howard University}
\email{jyeager@howard.edu}
\thanks{The author would like to express his deep gratitude to his Ph.D. advisor Karsten Grove, both for suggesting this problem, and for the many conversations we had in which he illuminated the technical issues.}

\footnotesize
\begin{abstract}
Let $M$ be a compact, connected and simply-connected Riemannian manifold, and suppose that $G$ is a compact, connected Lie group acting on $M$ by isometries. The dimension of the space of orbits is called the \emph{cohomogeneity} of the action. If the direct sum of the higher homotopy groups of $M$, tensored with the field of rational numbers, is a finite-dimensional vector space over the rationals, then $M$ is said to be \emph{rationally elliptic}. It is known that $M$ is rationally elliptic if it supports an action of cohomogeneity zero or one. When the cohomogeneity is two, this general result is no longer true. However, we prove that $M$ is rationally elliptic in the two-dimensional case under the added assumption that $M$ has nonnegative sectional curvature.
\end{abstract}
\normalsize

\maketitle



\section{Introduction}
Computing the higher homotopy groups $\pi_n (X)$ of a topological space $X$ is usually a difficult task. By tensoring these groups with the field of rationals $\textbf{Q}$, we remove the torsion elements and considerably simplify their computation. A simply-connected manifold $M$ is said to be \textit{rationally elliptic} if
\begin{center}$\sum_{n=2}^{\infty}$ dim$(\pi_n (M) \otimes \textbf{Q} )< \infty$.\end{center}
In the interest of brevity, we shall use the abbreviation elliptic for such manifolds. The ellipticity of $X$ is equivalent to the condition that the rational homology of the loop space is polynomially bounded; i.e., there is a constant $A >0$ such that
\begin{center}
$\sum_{i=0}^{n}$ dim$(H_i(\Omega X ; \textbf{Q}) \le An^r$,
\end{center} 
for $n \ge 1$ and fixed $r$.
F\'{e}lix, Halperin and Thomas present a complete treatment of these ideas in their book \cite{yFsHjcT}. The simplest example of such a space is the sphere $S^n$ for $n \ge 2$. It was shown by Serre \cite{jpS} that compact Lie groups are elliptic. It is a straightforward consequence of the exact homotopy sequence of the fiber bundle $G \rightarrow G/H$ that the same is true for their homogeneous spaces.
  
We shall consider a compact, connected Riemannian manifold $M$ on which a compact, connected  Lie group $G$ acts as a group of isometries. The space of orbits of this $G$-action can be made into a metric space in a natural way and we will denote this space by $M/G$. The dimension of the orbit space is called the \emph{cohomogeneity} of the action. Manifolds having actions of cohomogeneity zero or one are known to be elliptic. In the first case, the manifolds are just the homogeneous spaces mentioned above. The cohomogeneity one case is more difficult and was proved in a paper of Grove and Halperin \cite{kGsH}.

It has been conjectured that every simply connected Riemannian manifold of nonnegative sectional curvature is rationally elliptic~\cite{sCpYkGjW}. In our work we shall prove this conjecture for manifolds supporting a cohomogeneity two action.

\begin{MAINthm}\label{T:MAIN}
Let $M$ be a compact, connected and simply-connected Riemannian manifold having nonnegative sectional curvature. Let $G$ be a compact, connected Lie group acting effectively on $M$ by isometries, with orbit space $M/G$. If $dim(M/G) = 2$, then $M$ is rationally elliptic. 
\end{MAINthm}

The proof utilizes two somewhat different techniques. Most of the orbit spaces considered can be given a metric of constant sectional curvature, either 0 or 1. These generate tilings of either the plane $\textbf{R}^2$ or the sphere $S^2$, respectively. For these cases we shall use the Morse theory and the Serre spectral sequence to prove our theorem. For the remaining cases, which are not amenable to this treatment, the result will follow almost directly from a theorem of K. Grove and S. Halperin in the paper referenced above.

\section{Preliminaries}
We shall assume from here on that our manifold $M$ and Lie group $G$ are compact and connected and that $\pi_1(M) = 0$. \textit{Curvature} always signifies sectional curvature, whenever the noun is used without a modifier.

\subsection{Fibrations, Homotopy Fiber, Loop Spaces}\label{Ss:FIBRATION}
The following basic material is well-known and is taken directly from Hatcher \cite{aH}. The results that we shall need for our work are mostly direct consequences of the concept of the homotopy fiber of a mapping. If $f: E \rightarrow B$ is a mapping of topological spaces, we define the space 
\begin{center}
$E_f = \{(x,\gamma)\  | \  x \in E, \gamma : [0,1] \rightarrow B, \gamma(0) = f(x)\}$
\end{center}
and the mapping $p : E_f \rightarrow B$ by $p(x,\gamma) = \gamma(1)$. Then $p$ is a fibration and there is a homotopy equivalence $E_f \cong E$. If $b \in B$, the homotopy fiber at $b$ is the space 
\begin{center}
$F_b = \{(x,\gamma)\  | \ x \in E$, $\gamma : [0,1] \rightarrow B$, $\gamma (0) = f(x)$, $\gamma (1) = b\}$.
\end{center}
If $p : E \rightarrow B$ is already a fibration, we have the following theorem.

\begin{thm}
If $p : E \rightarrow B$ is a fibration, then the inclusion $E \rightarrow E_p$ is a fiber homotopy equivalence; i.e., the homotopy fibers are homotopy equivalent to the actual fibers.
\end{thm}

Now let  $p : E \rightarrow B$ be a fibration with fiber $F = p^{-1}(b_0)$. Choose a basepoint $e_0 \in F$. The homotopy fiber consists of pairs $(e,\gamma)$ with $e \in E$ and $\gamma$ a path from $p(e)$ to $b_0$. The inclusion $F \rightarrow F_p$ is a homotopy equivalence and extends to a map $i : F_p \rightarrow E$ where $i(e,\gamma) = e$. Clearly, this mapping is a fibration. The fiber $F_i\ =\ i^{-1}(e_0)$ consists of pairs $(e_0,\gamma)$ with $\gamma$ a loop in $B$ at the basepoint $b_0$. This fiber is just $\Omega B$ so there is a homotopy equivalence between $\Omega B$ and $F_i$. It is not difficult to construct a retraction of $F_i$ onto $\Omega B$. The composition
\begin{center}
$\Omega B \rightarrow F_i \rightarrow F_p \rightarrow F$
\end{center} 
where the map $F_p \rightarrow F$ is a homotopy inverse of the inclusion gives a map $\Omega B \rightarrow F$ which is clearly homotopic to a fibration. The following diagram summarizes the situation.

\[
\begin{CD}
@....	@>>>	F_j			@>i>> 			F_i			@>j>>	F_p			@>i>>	E		@>p>>	B\\
@.	@.		@AA\cong A					@AA\cong A			@AA\cong A			@|				@|\\
@....	@>>>	\Omega E		@>\Omega p>>	\Omega B		@>>>	F			@>>>	E		@>p>>	B
\end{CD}
\] \bigskip

This process can be iterated indefinitely showing that there is an infinite sequence of loop spaces
\begin{center}
$...\rightarrow \Omega^k F \rightarrow \Omega^k E \rightarrow \Omega^k B \rightarrow ...  \rightarrow \Omega F \rightarrow \Omega E\rightarrow \Omega B \rightarrow F \rightarrow  E \rightarrow B$
\end{center}
where the consecutive maps are fibrations up to homotopy equivalence.

\subsection{Group Actions and Orbit Spaces}
Let $G \times M \rightarrow M$ be an effective isometric action of a compact, connected Lie group on a compact, connected Riemannian manifold. For $p \in M$, the isotropy group $G_p$ is the subgroup of $G$ fixing the point $p$. The orbit of $p$ is the homogeneous space $Gp = G/G_p$. The set of orbits can be given the structure of an Alexandrov space, denoted by $M/G$. An action is \textit{polar} if there is a complete immersed submanifold $S$ of $M$ which meets all orbits of G orthogonally. $S$ is called a section of the action and is always a totally geodesic submanifold. The action is called \textit{hyperpolar} if, in addition, the section is flat as, for example, in the well-known adjoint action of a Lie group on itself. 

\subsection{Grove-Halperin Theorem}
One of the main tools that we shall employ is a theorem of Grove and Halperin. In \cite{kGsH} they show that a compact simply-connected Riemannian manifold $M$ supporting a cohomogeneity one action of a compact Lie group $G$ is necessarily elliptic. An orbit space of such an action must be either a circle or an interval. The case of the circle is of no importance for us since, in that case, $M$ will not be simply-connected. Denote the endpoints of the interval by $a$ and $b$. Choose some point $e$ in the interior of the interval. Let $A, B, E$ denote the orbits in $M$ corresponding to these points. Since these are all homogeneous spaces of $G$, they are elliptic. To prove their result, they note the fact that there are disk bundles $D(A)$ and $D(B)$ in $M$ over $A$ and $B$, respectively, having $E$ as their common boundary. The idea then is to create a "double mapping cylinder" to represent $M$ as a disjoint union
\begin{center}$A\  \amalg\  (E \times I)\  \amalg \ B$\end{center}
with attaching maps $\phi_0 : E \times \{0\} \rightarrow A$ and $\phi_1 : E \times \{1\} \rightarrow B$, where $I$ is the closed interval [0,1]. By analyzing this cylinder for the various possible cohomogeneity one actions, they prove that $M$ is elliptic.

The proof they give in their paper is carried out in greater generality, and does not actually require that $A$, $B$ and $E$ be homogeneous spaces, but only that they be elliptic. The cohomogeneity one property was used merely to establish the ellipticity of these spaces and to show the existence of the disk bundles. Therefore, if we know that $A$, $B$ and $E$ are elliptic, and if we can establish the existence of the disk bundles $D(A)$ and $D(B)$ in some other way, their basic argument will apply without modification. In view of these remarks, we restate their results in a form more suited to our needs. 

\begin{GHthm}\label{T:GH}
Let $M$ be a compact, connected and simply-connected Riemannian manifold. Suppose that $A$ and $B$ are disjoint submanifolds of $M$ and that $D(A)$ and $D(B)$ are disk bundles in $M$ over $A$ and $B$, respectively, having the submanifold $E$ as their common boundary. Then $M$ is elliptic if and only if  $E$ (or $A$ or $B$) is elliptic.
\end{GHthm}

\subsection{Riemannian Submersions}\label{Ss:Foliation}
The basic facts about Riemannian submersions can be found in O'Neill's paper~\cite{bO}. Let $M$ and $B$ be Riemannian manifolds and let $\pi : M \rightarrow B$ be a map. Define the vertical subspace $V_p = \{v \in T_p\ |\ \pi_*v = 0\}$. The orthogonal complement of $V_p$, denoted by $H_p$, is called the horizontal subspace. The map $\pi$ is called a Riemannian submersion if $\pi$ has maximal rank and if $\pi_*$ preserves the lengths of horizontal vectors. If $X$ is a vector field on $M$, denote by $vX$ and $hX$ the vertical and horizontal components of $X$, respectively. Define a tensor $A_X Y = v\triangledown_{hX}hY + h\triangledown_{hX}vY$ where $\triangledown$ is the Levi-Civita connection. When $X,Y$ are horizontal, it is easy to show that $A_X Y = \frac{1}{2}v[X,Y]$ where $[X,Y]$ is the Lie bracket of $X$ and $Y$. Thus $A_X Y$ is a measure of the non-integrability of the horizontal distribution. Let $X, Y$ denote an orthonormal pair of horizontal vectors at the point $p$, and let $K$ denote the sectional curvature in the plane of $X$ and $Y$. If $K'$ denotes the sectional curvature in $B$ corresponding to $\pi_* X$ and $\pi_* Y$, then $K = K' - 3|A_X Y|^2$ and it follows that $K' \ge K$.

\subsection{Morse Theory}\label{Ss:MORSE}
Let $M$ be a Riemannian manifold and let $P$ and $Q$ be two submanifolds of $M$, where we do not rule out the possibility that $P = Q$, although we shall only be interested in the case where $P$ is a point. The space of all piecewise smooth paths in $M$ from $P$ to $Q$ will be denoted by $\Omega (M;P,Q)$. We define the energy $E$ on $\Omega (M;P,Q)$ by

\begin{center}
$E(\omega) = $ \Large $\int_{0}^{1}\ \vline {\frac{d\omega}{dt}}\vline\ ^2$ \normalsize $dt$.
\end{center}

\noindent As is well known from the calculus of variations, the first variation of $E$ vanishes; i.e., $\lambda$ is a critical point of $E$, precisely when $\lambda$ is a horizontal geodesic.  At a critical point of $E$, the Hessian $E_{**}$ is well-defined. We shall be interested in computing the dimension of its nullspace at $\gamma$, which is always finite.  For this we need to look at a second variation of $E$. We find that the nullspace of $E_{**}$ consists of certain vector fields defined along the critical geodesics. These are the Jacobi fields and they satisfy the differential equation
\begin{center}
\Large $\frac{D^2 J}{dt^2}$ \normalsize $+ R(\gamma^{'}(t),J)\gamma^{'}(t) = 0$,
\end{center} 
where \Large $\frac{D}{dt} = $ \normalsize $\nabla_{\gamma {'}(t)}$ is the covariant derivative along $\gamma$ and $R$ is the Riemann curvature tensor.
Let $Q_q$ denote the tangent space to $Q$ at the point $q \in Q$. Suppose that $\gamma(t)$ is a geodesic defined on [0,T] such that $\gamma(0) \perp Q_q$. A $Q$-Jacobi field is a Jacobi field which is orthogonal to $\gamma$ with $J(0) \in Q_{\gamma(0)}$ and $J'(0) - S_0 J(0) \perp Q_{\gamma(0)}$, where $S_t$ is the second fundamental form of $Q$ at $\gamma(t)$ with respect to $\gamma'(t)$. A $Q$-focal point is a point $\gamma(t)$, $t \in (0,T]$, for which there is a non-zero $Q$-Jacobi field which vanishes at $\gamma(t)$. The multiplicity of the the $Q$-focal point is the dimension of such Jacobi fields. Let $S_T$ be the second fundamental form of $Q$ at $\gamma(T)$. Define a symmetric bilinear map $A$ on the space spanned by these Jacobi fields 
\begin{center}
$A(J_1,J_2)\  = \ <J_1'(T) - S_T J_1(T),J_2(T)>$
\end{center}
whose value at T is contained in $Q_{\gamma(T)}$. We have the following extensions~\cite{dK1},~\cite{dK2} of the usual Index Theorem and Morse Theorem to the cases where the ends are submanifolds.

\begin{thm}\label{T:EXTINDEX}
If $\gamma : [0,T] \rightarrow M$ is a critical geodesic and $\gamma(T)$ is not a focal point of $Q$, then the index of $\gamma$ is the number of $Q$-focal points $\gamma(t)$, with $0 < t < T$, each counted with its multiplicity, plus the index of $A$ at $T$.
\end{thm}

\begin{thm}\label{T:EXTMORSE}
Let $M$ be a complete Riemannian manifold and let $P$, $Q$ be two submanifolds of $M$ which do not contain focal points of each other along any geodesic. Then $\Omega(M;P,Q)$ has the homotopy type of a countable CW-complex which contains one cell of dimension $\lambda$ for each horizontal geodesic from $P$ to $Q$ of index $\lambda$.
\end{thm}

\subsection{Two-Dimensional Orbit Spaces}
In this section we assume additionally that dim$(M/G) = 2$, and that curv($M) \ge 0$. We begin by showing that $M/G$ is simply-connected, severely limiting the possibilities for the orbit space. We note that a closed curve in $M/G$ can be uniformly approximated by a broken geodesic curve in the same homotopy class.

\begin{thm}\label{T:S-C}
$M/G$ is simply-connected.
\begin{proof}
Let $\gamma(t)$ be a closed curve in $M/G$ with basepoint $p$. As noted above, we can assume that $\gamma$ is piecewise geodesic, without changing its homotopy class. Then we can lift $\gamma(t)$ (not uniquely) to a curve $\tilde{\gamma}(t)$ with both initial and terminal points in $\pi^{-1}(p)$. Since all orbits are arcwise connected, we can join these points by a curve $\alpha$ lying entirely in their common orbit. Since $M$ is simply-connected, there is a homotopy of $\tilde{\gamma}\alpha$ to a point. The projection of this homotopy is the desired homotopy of $\gamma(t)$ to a point.\qedhere
\end{proof}
\end{thm}

The set of conjugacy classes of isotropy groups form a lattice having a smallest element [$H$]. Orbits whose isotropy groups are conjugate to $H$ are called \textit{principal} orbits. The principal orbits form an open dense submanifold of $M/G$. Non-principal orbits having the same dimension as a principal orbit are called \textit{exceptional} orbits. The remaining orbits are called \textit{singular} orbits. The orbit space is stratified by this lattice. There is a single 2-dimensional stratum corresponding to the principal isotropy group. The other orbit strata are either 0 or 1-dimensional. These orbit strata are totally geodesic submanifolds of $M/G$. Combining theorems (3.12, 4.4, 4.5, 8.3, 8.6) in chapter IV of Bredon~\cite{gB}, we have the following useful theorem.

\begin{thm}\label{T:BREDON}
Suppose that $M$ is simply-connected and supports a cohomogeneity two action.\newline
a) If $\partial M/G \ne \varnothing$, the boundary consists entirely of singular orbits and there are no exceptional orbits.\newline
b) If $\partial M/G = \varnothing$, the exceptional orbits in $M/G$ are isolated.
\end{thm}

The map $\pi : M \rightarrow M/G$, when restricted to principal orbits, is a Riemannian submersion. We shall need to consider the case where curv$(M/G) = 0$, $\partial M/G = \varnothing$, and the only non-principal orbits are exceptional. Since we are assuming that curv$(M) \ge 0$, this implies that curv$(M) = 0$. If $X$ and $Y$ are horizontal vectors, it follows that $A_X Y = 0$; i.e., that $[X,Y]$ is horizontal. Then the horizontal distribution is integrable, and there is an immersed submanifold $S$ that is orthogonal to each of the vertical subspaces. It is known in this case that the action is polar~\cite{hB}, and we have the following theorem of Alexandrino and T\"{o}ben~\cite{mAdT}.

\begin{thm}\label{T:NOEXCEPTIONAL}
A non-trivial polar action on a simply-connected manifold has no exceptional orbits.
\end{thm}

\noindent Grove and Ziller~\cite{kGwZ} show that in this case, when there are no singular orbits, that $M = \textbf{R}^2 \times G/H$ where $G/H$ is a principal orbit. But then $M$ is not compact, a contradiction. For emphasis, we summarize this discussion in the following theorem.

\begin{thm}
No flat two-dimensional orbit space is homeomorphic to a sphere.
\end{thm}

We shall later need to consider triangles in $M/G$ whose sides are geodesics. As a straightforward consequence of the Gauss-Bonnet theorem, we have the following basic theorem regarding geodesic triangles.

\begin{thm}\label{T:ANGLESUM}
Suppose that dim$(M/G)=2$ and that $\alpha,\beta,\gamma$ are the angles of a geodesic triangle in $M/G$ whose interior consists entirely of principal orbits.~If curv$(M/G) > 0$, then $\alpha+\beta+\gamma > \pi$. If curv$(M/G) = 0$, then $\alpha+\beta+\gamma = \pi$.
\end{thm}

\section{Two-Dimensional Orbit Spaces of Non-Negative Curvature}

We continue to assume that $M$ is a compact, connected and simply-connected manifold of nonnegative curvature, and that dim$(M/G) = 2$. We shall classify all possible two dimensional orientable orbit spaces that can admit a metric of nonnegative curvature in the sense of Alexandrov. The points of $M/G$ corresponding to principal orbits form a smooth manifold, with boundary either empty or consisting of closed geodesic arcs and exceptional orbits. It is well-known that the angle between two boundary arcs must be an angle in the set $\{\pi/2$, $\pi/3$, $\pi/4$, $\pi/6\}$. In the case where $\partial M/G = \varnothing$, the interior may also contain isolated singular points; i.e., the exceptional orbits. At an exceptional orbit, the tangent space is known to be a cone, subtending a solid angle of $2\pi/p$, for some integer $p \ge 2$.

\subsection{Positive Curvature}
In this section we shall eliminate those orbit spaces that cannot carry a positive curvature metric, and show that the remaining ones do support such a metric. Many of these can be given a metric of constant positive curvature and can generate tilings of the 2-sphere. Our method of showing that an orbit space does not support positive curvature is quite simple. We triangulate the orbit space in some way, obtain the sum of all angles for all triangles, and divide by the number of triangles. If the average is less than or equal to $\pi$, we conclude that at least one of the triangles must have an angle sum less than or equal to $\pi$. Our first step in the classification is to establish limitations on the number of singular points and boundary arcs. 

\begin{thm}\label{T:SING}
Let $M/G$ be an orientable 2-dimensional orbit space having positive curvature in the sense of Alexandrov. Then\newline
a) $\partial M/G$ cannot consist of more than 3 boundary arcs.\newline
b) If $\partial M/G = \varnothing$, then there can be at most 3 exceptional orbits.\newline
c) If $\partial M/G$ consists of 1 boundary arc, then there can be at most 1 singular boundary point. If $\partial M/G$ consists of 2 or 3 boundary arcs, then there will be 2 or 3 singular boundary points, respectively. In any case, there are no exceptional orbits. 
\end{thm}

\begin{proof}

a): Suppose that there are $k > 3$ boundary arcs and, therefore, $k$ singular boundary points. Choose one of these points $p$ and construct geodesic arcs to each of the other singular boundary points not adjacent to $p$. This construction creates $k - 2$ triangles for which the total angle sum is at most $k\pi/2$. The average angle sum for a triangle is \Large $\frac{\pi}{2} \cdot \frac{k}{k - 2}\ $\normalsize  $\le \pi$, if $k \ge 4$.

b): Suppose that there are at least four exceptional orbits. Choose any three of them and connect the points pairwise by minimizing geodesics, thus forming a geodesic triangle. Connect the fourth point to each of the other three points by minimizing geodesics. This construction creates four triangles with a total angle sum of at most $4\pi$.

c): If the boundary consists of a single arc, there can be at most one boundary singular point (if the ends of the arc intersect at an angle less than $\pi$). Otherwise, there are as many singular boundary points as the number of arcs. By theorem~\ref{T:BREDON}, there are no exceptional orbits.\qedhere
\end{proof}

It is convenient to divide the classification into two cases, depending on whether the boundary of $M/G$ is empty or not. We assume hereafter that $\alpha, \beta, \gamma$ are angles in the set $\{\pi/2, \pi/3, \pi/4, \pi/6\}$ and that $p, q, r$ are integers $\ge 2$.

\subsubsection{Empty Boundary}
When $\partial M/G = \varnothing$, $M/G$ is topologically $S^2$ but may have interior singular points. We have seen that there can be at most three such points. These are orbifold points resulting from an orientable $\textbf{Z}_p$ action, subtending an angle of $2\pi/p$, for $p \ge 2$. 

\noindent \textbf{No Singular Point:} This is the case where $M/G = S^2$ and trivially supports a metric of constant positive curvature.

\noindent \textbf{One Singular Point:} When there is only one singular point, the angle has a value  $2\pi/p$. This orbit space is sometimes called the $\textbf{Z}_p$ teardrop and is known to arise from an $S^1$ action on $S^3$. While this space has positive curvature, it cannot be given a metric of constant positive curvature.

\noindent \textbf{Two Singular Points:} If there are two singular points, let the angles be $2\pi/p$ and $2\pi/q$. This orbit space is known as the  $\textbf{Z}_p$-$\textbf{Z}_q$ football and, like the $\textbf{Z}_p$ teardrop, supports positive curvature but not constant positive curvature, unless $p=q$.

\noindent \textbf{Three Singular Points:} Finally, we consider the case of three singular points with angles $2\pi/p$, $2\pi/q$ and $2\pi/r$. We can view this space as the double of the triangle having angles $\pi/p$, $\pi/q$ and $\pi/r$. There are several possibilities for $p$, $q$ and $r$. One set of possibilities is $p = q = 2$ and $r \ge 2$. The others are $p=2$, $q=3$ and $r=3,\ 4$ or $5$. Each of these spaces can be realized as a spherical triangle in $S^2$ and, consequently, can be given a constant positive curvature metric.

\subsubsection{Non-Empty Boundary}
When $\partial M/G \ne \varnothing$, the boundary will consist of from one to three arcs and, by theorem~\ref{T:BREDON}, will not contain an interior singular point. 

\noindent \textbf{One Boundary Arc:} If the boundary has no singular points the orbit space can be given the metric of a closed hemisphere of $S^2$. If the boundary has one singular point with angle $\alpha$, the the double of $M/G$ is the $\textbf{Z}_p$ teardrop with $p = \pi/\alpha$ and $M/G$ supports positive but not constant curvature.

\noindent \textbf{Two Boundary Arcs:} Let the boundary angles be denoted by $\alpha$ and $\beta$. The double of $M/G$ is the $\textbf{Z}_p$-$\textbf{Z}_q$ football where $p = \pi/\alpha$ and $q = \pi/\beta$. Therefore, $M/G$ supports positive but not constant curvature (unless $\alpha = \beta$).

\noindent \textbf{Three Boundary Arcs:} Let the boundary angles be denoted by $\alpha$, $\beta$ and $\gamma$. If $\alpha$, $\beta = \pi/2$, then $\gamma$ can be any angle in the set $\{\pi/2, \pi/3, \pi/4, \pi/6\}$. If $\alpha = \pi/2$ and $\beta = \pi/3$, then $\gamma = \pi/3$ or $\pi/4$. No other combinations are possible and all of these cases can be realized as spherical triangles.

\subsection{Zero Curvature}
We shall now consider the case where the manifold does not have strictly positive curvature. The following theorem shows that it is sufficient to consider only flat metrics on the orbit space.

\begin{thm}\label{T:FLAT}
Let $M/G$ be an orientable 2-dimensional orbit space having nonnegative but not strictly positive curvature in the sense of Alexandrov.\newline
a) If $\partial M/G \ne \varnothing$, then we may assume that $\partial M/G$ consists of 3 or 4 boundary arcs.\newline
 b) If $\partial M/G = \varnothing$, then $M/G$ contains a point of positive curvature and the orbit space can be given a metric of strictly positive curvature.
\end{thm}

\begin{proof}
a): Suppose that there are $k > 4$ boundary arcs and, therefore, $k$ singular boundary points. Choose one of these points $p$ and construct geodesic arcs to each of the other singular boundary points not adjacent to $p$. This construction creates $k - 2$ triangles for which the total angle sum is at most $k\pi/2$. The average angle sum for a triangle is \Large $\frac{\pi}{2} \cdot \frac{k}{k - 2}\ $\normalsize  $< \pi$, if $k > 4$. Orbit spaces with one or two boundary arcs, if they exist, are essentially the same as similar orbit spaces described in the section on positive curvature, and a proof of the main theorem for all of these spaces will be given that does not depend on the curvature of the space.

b): By the results of section~\ref{Ss:Foliation}, the orbit space must have a point of positive curvature. It cannot contain four exceptional orbits since, in this case, all of the angles at the exceptional orbits must be $\pi$, and the orbit space must be a flat square. Under our assumptions, there can be only three exceptional orbits. Connect them with geodesics to create two triangles. At least one of these triangles encloses a point of positive curvature and, therefore, must have an angle sum exceeding $\pi$. Since all angle sums for a triangle must be at least $\pi$, the average sum is greater than $\pi$ and, as in the section on positive curvature, the orbit space can be given a metric of constant positive curvature.\qedhere
\end{proof}

\noindent It follows that, if the orbit space cannot be given a metric of positive curvature, it must be a planar triangle or a square. There are only three such triangles, $(\pi/2,\pi/4,\pi/4)$, $(\pi/2,\pi/3,\pi/6)$ and $(\pi/3,\pi/3,\pi/3)$.

\subsection{Tabular Listing}

The previous results are summarized in tables~\ref{Ta:PosCurv} and~\ref{Ta:ZeroCurv}. The exterior and interior angles are listed for each orbit space. In case 8, where there is one boundary arc but no angle (a circle), the angle is entered as $\pi$ to indicate that there is a boundary. When the orbit space can be represented as a spherical triangle, it is interesting to determine whether this triangle can be used to produce a tiling of the whole sphere. This question will be studied in the next section. The second table gives a similar classification of orbit spaces that can tile the plane. In this case, the number of tiles will be infinite.

We continue to use the convention that Greek letters denote boundary angles in the set $\{\pi/2, \pi/3, \pi/4, \pi/6\}$, while latin letters represent integers $\ge 2$.

\begin{table} [!h]
	\begin{center}
		\begin{tabular}{| c | c | c | c | c | c |}
 			\hline
			\# & Ext $\angle$'s & Int $\angle$'s & Tiling? & Tile & No. Tiles\\ \hline
			1 &  &  & Yes & $S^2$ & 1\\ \hline
			2 &  & $2\pi/p$ & No &  & \\ \hline
			3 &  & $2\pi/p$, $2\pi/q$ & No & & \\ \hline
			4 &  & $\pi$, $\pi$, $2\pi/p$ & Yes & $(\pi/2,\pi/2,\pi/p)$ & $4p$\\ \hline
			5 &  & $\pi$, $2\pi/3$, $2\pi/3$ & Yes & $(\pi/2,\pi/3,\pi/3)$ & 24\\ \hline
			6 &  & $\pi$, $2\pi/3$, $\pi/2$ & Yes & $(\pi/2,\pi/3,\pi/4)$ & 48\\ \hline
			7 &  & $\pi$, $2\pi/3$, $2\pi/5$ & Yes & $(\pi/2,\pi/3,\pi/5)$ & 120\\ \hline
			8 & $\pi$ &  & Yes & Hemisphere & 2\\ \hline
			9 & $\alpha$ &  & No &  & \\ \hline
			10 & $\alpha$, $\beta$ &  & No &  & \\ \hline
			11 & $\pi/2$, $\pi/2$, $\alpha$ &  & Yes & $(\pi/2,\pi/2,\alpha)$ & $4\pi/\alpha$\\ \hline
			12 & $\pi/2$, $\pi/3$, $\pi/3$ &  & Yes & $(\pi/2,\pi/3,\pi/3)$ & 24\\ \hline
			13 & $\pi/2$, $\pi/3$, $\pi/4$ &  & Yes & $(\pi/2,\pi/3,\pi/4)$ & 48\\ \hline
		\end{tabular}
		\caption{Orbit Spaces Admitting Positive Curvature}\label{Ta:PosCurv}
	\end{center} 
\end{table}

\begin{table} [!h]
	\begin{center}
		\begin{tabular}{| c | c | c | c | c | c |}
 			\hline
			\# & Ext $\angle$'s & Int $\angle$'s & Tiling? & Tile\\ \hline
			14 & $\pi/2$, $\pi/3$, $\pi/6$ & & Yes &$(\pi/2,\pi/3,\pi/6)$\\ \hline
			15 &  $\pi/2$, $\pi/4$, $\pi/4$ & & Yes &$(\pi/2,\pi/4,\pi/4)$\\ \hline
			16 &  $\pi/3$, $\pi/3$, $\pi/3$ & & Yes &$(\pi/3,\pi/3,\pi/3)$\\ \hline
			17 &  $\pi/2$, $\pi/2$, $\pi/2$, $\pi/2$ & & Yes &$(\pi/2,\pi/2,\pi/2,\pi/2)$\\ \hline
		\end{tabular}
		\caption{Orbit Spaces Admitting Zero Curvature}\label{Ta:ZeroCurv}
	\end{center} 
\end{table}

\subsection{Spherical and Planar Tilings}
In this section we shall see that many of the orbit spaces, those that support constant curvature, can be made to tile either the sphere or the plane. We have seen that, except for the trivial cases of the sphere and hemisphere, all of the orbit spaces $M/G$ supporting constant positive curvature can be realized as spherical triangles, or as spaces obtained by gluing along some or all of the edges of such triangles. These are the cases 1, 4-8 and 11-13 in table~\ref{Ta:PosCurv}. We will show that each triangle generates a tiling of the sphere. The orbit spaces 14-17 obviously generate a tiling of the plane.

\noindent \textbf{Cases 1, 8:}  (Sphere and hemisphere): These cases clearly tile the sphere.

\noindent \textbf{Cases 4, 11:} ($\pi/2,\pi/2,\theta$): In these cases $\theta$ divides $2\pi$ evenly with quotient $k$. Therefore, $2k$ triangles will tile the whole sphere.

\noindent \textbf{Cases 5, 12:} ($\pi/2,\pi/3,\pi/3$): Reflect over an edge joining angles $\pi/2$ and $\pi/3$. This produces a triangle with angles $2\pi/3$, $\pi/3$ and $\pi/3$. Reflecting twice around the angle $2\pi/3$ produces a triangle with all three angles equal to $2\pi/3$. This is a fundamental region for the action of the tetrahedral group on $S^2$, and so gives a tiling of the sphere with 24 tiles.

\noindent \textbf{Cases 6, 13:} ($\pi/2,\pi/3,\pi/4$): Reflecting the triangle over the edge connecting the angles $\pi/2$ and $\pi/4$ produces the triangle of case 5. The tiling consists of 48 tiles.

\noindent \textbf{Case 7:} ($\pi/2,\pi/3,\pi/5$): Reflect around the angle $\pi/5$ ten times to remove this angle. When a reflection involves the angle $\pi/2$, the vertex is removed. There remain five vertices, each with angle $2\pi/3$. This pentagon is a fundamental region for the action of the dodecahedral group on the 2-sphere, and the triangle tiles the sphere with 120 tiles.


\section{Main Theorem}

We shall now prove our main theorem. We restate the theorem for the convenience of the reader.

\begin{MAINthm}
Let $M$ be a compact, connected and simply-connected Riemannian manifold having nonnegative sectional curvature. Let $G$ be a compact, connected Lie group acting effectively on $M$ by isometries, with orbit space $M/G$. If $dim(M/G) = 2$, then $M$ is rationally elliptic. 
\end{MAINthm}

\noindent The proof will be accomplished in two parts. All of the orbit spaces that can be given a metric of constant nonnegative curvature can be used to tile either the sphere or the plane. We shall use these tilings, together with the Morse theory, to prove our theorem. For those orbit spaces which do not accept a metric of constant nonnegative curvature, we shall prove the theorem as an application of the Grove-Halperin theorem.

\subsection{Constant Curvature}\label{S:CONSTANT}
In all of the cases of constant curvature, the orbit spaces generate tilings of $\textbf{R}^2$ or $S^2$. By a general theorem of Mendes~\cite{rM}, the constant curvature metric on $M/G$ can be lifted to $M$. 

\begin{thm} (Mendes). Let $M$ be a polar $G$-manifold with section $\Sigma$, and $W(\Sigma)$ the generalized Weyl group associated to $\Sigma$. Let $\sigma \in C^{\infty}(Sym^2(T\ast \Sigma))^{W (\Sigma)}$ be a $W(\Sigma)$-equivariant symmetric two-tensor on the section $\Sigma$. Then there exists a $G$-equivariant symmetric two-tensor $\tilde{\sigma}$ on $M$ which restricts to $\sigma$. Moreover, if $\sigma$ is a metric, that is, positive definite at every point, then we can choose $\tilde{\sigma}$ on $M$ to be a metric with respect to which the action of $G$ remains polar with the same sections.
\end{thm}

\noindent It is interesting to note that the new metric on $M$ has the constant curvature orbit space but may no longer have nonnegative curvature. However, this property will no longer be needed in the sequel, since its only use was to limit the geometry of the orbit space.

When $Q$ is a principal orbit of $M$, we have an inclusion $i: Q \rightarrow M$ and, as described in sec.~\ref{Ss:FIBRATION}, we can construct a fibration $Q_i \rightarrow M$ with homotopy fiber $F$ at a point $p$ in $M$. We will need to show that $F$ has polynomially bounded homology. If the orbit $Q$ is simply-connected, it will follow from the sequence of loop spaces described at the end of sec.~\ref{Ss:FIBRATION} and the Serre spectral sequence, that the same is true for $M$ since the homology groups $H_*(F;\textbf{Q}) \otimes H_*(\Omega Q;\textbf{Q})$ converge to $H_*(\Omega M;\textbf{Q})$. Even if the fundamental group of $Q$ is non-zero, but finite, we may still use this method by considering instead the universal covering space $\tilde{Q}$.

When $\pi_1(Q)$ is infinite, this method breaks down. Fortunately, we have the following workaround, which is due to Grove and Ziller. By representation theory, we can embed $G$ in $SU(N)$ for some $N$. Since $G$ acts on $M$, there is a bundle with fiber $M$ associated to the principal bundle $SU(N) \rightarrow SU(N)/G$ given by
\begin{center}$(SU(N) \times M)/G \rightarrow SU(N)/G$.\end{center}
Denote the total space of this bundle by $P$. Then the exact homotopy sequence of this bundle, tensored with the rationals, shows that $P$ is elliptic if and only if $M$ is elliptic. The space $P$ supports an $SU(N)$ action by multiplication on the first factor, also with orbit space $M/G$. Let $SU(N)/K$ be a principal orbit. The homotopy sequence of the principal bundle $SU(N) \rightarrow SU(N)/K$ ends with
\begin{center} ... $\rightarrow \pi_1(SU(N)) \rightarrow \pi_1(SU(N)/K) \rightarrow \pi_0(K) \rightarrow \pi_0(SU(N)) \rightarrow \pi_0(SU(N)/K)$.\end{center}
Since $\pi_1(SU(N))$ and $\pi_0(SU(N))$ are trivial, and since $K$ has a finite number of components, it follows that $\pi_1(SU(N)/K)$ is a finite group. For simplicity of exposition, we shall continue to work directly with the manifold $M$ and, by the previous remarks, we might as well assume that a principal orbit of $M$ is simply connected.

The homotopy fiber $F$ is the set of all paths $\gamma : [0,1] \rightarrow M$ with $\gamma(0) \in Q$ and $\gamma(1) = p$. It is convenient for our purposes to consider the paths as beginning at $p$ and ending on $Q$. Then $\gamma(0) = p$ and $\gamma(1) \in Q$. We shall use the Morse theory to compute the homology of $F$. The critical points of the energy functional are the horizontal geodesics beginning at the point $p$ and orthogonal to the orbit $Q$. Since $M/G$ tiles the plane or the sphere, we can view the orbit space as multiply embedded in $\mathbf{R}^2$ or $S^2$, and we may then regard each horizontal geodesic in M as a curve over a geodesic in that space. Choose one of the tiles to be the primary orbit space and consider the point $p$ to be in this tile. The critical geodesics correspond to geodesics in the tiled space that connect $p$ to a copy of $Q$. The methods we use are similar for positive and zero curvature, but the details are somewhat different, and so we shall discuss these cases separately. 

\subsubsection{Positive Curvature}
Let $Q$ be any principal orbit, and let $C$ denote the union of all great circles in $S^2$ that contain two or more copies of $Q$. Let $P$ in $M$ - $\pi^{-1}(C)$ be a principal orbit such that, for some $p \in P$, $p$ is not a focal point of $Q$ along any geodesic. Let $p'$ denote the projection of $P$ in $M/G$. In what follows we denote the antipodal point of $p'$ by $\overline{p'}$ and the number of tiles in the spherical tiling by $c$ (from the No. Tiles column in table~\ref{Ta:PosCurv}). All horizontal geodesics emanating from $p$ lie over great circles passing through $\overline{p'}$. Since the dimension of the horizontal subspace is 2, the contribution to the index of a critical geodesic at each crossing of the antipodal point must be 1.

Under our assumptions, there are exactly $2c$ critical geodesics that meet the orbit $Q$ only once ($c$ when the orbit space is the double of a spherical triangle). These correspond to the great circles on $S^2$ from $p'$ to a point equivalent to $Q$. We must count two critical geodesics for each copy of $Q$, one for each direction from $p$. Every other critical geodesic must begin as one of these. Choose one of these initial segments.  Suppose that, on the corresponding great circle, $\gamma$ is a geodesic having index $d$, thereby producing a cell of dimension $d$. We must consider the contribution of the index of the operator $A$ (sec.~\ref{Ss:MORSE}) associated to the second fundamental form of $Q$, which may be a negative number. Set $m$ equal to the smallest possible value of index($A$). If $m \ge 0$ then, after a single circuit, there will be no additional cells of dimension $d$, only a cell of higher dimension, since there is a nonzero contribution at the antipodal point. Therefore, no two extensions of the same initial segment have the same index. If $\beta_d$ is the $d$-dimensional Betti number of $F$, it follows that $\beta_d \le 2c$. If $m < 0$, we may require as many as $1-m$ complete circuits to obtain a larger value for the index of $\gamma$. Then $\beta_d \le 2(1-m)c$. Setting $\lambda = max(1,1-m)$ and summing the Betti numbers, we have $\sum_{i=0}^{n} \beta_i \le 2c\lambda(n+1)$ and $F$ has polynomially, in fact linearly, bounded homology. This proves the main theorem for positive curvature.

\subsubsection{Zero Curvature}
Let $Q$ be any principal orbit, and let $P$ be a principal orbit such that, for some $p \in P$, $p$ is not a focal point of $Q$ along any geodesic. Since there is no antipodal point in zero curvature, we must look elsewhere for focal points. A focal point will occur when a critical geodesic crosses an orbit that is more singular than a principal orbit. More precisely, we have the following lemma.

\begin{lem}\label{T:CROSSING}
Let $Q$ be a principal orbit with isotropy group $H$, and let $S$ be any other orbit with isotropy group $K \supseteq H$. Suppose that $\gamma(t)$ is a horizontal geodesic joining the points $s \in S$ and $q \in Q$. Define $\gamma_k(t) = k\gamma(t)$ for $k \in K$. Then $\gamma_k(t)$ is a horizontal geodesic and $\gamma_k(t) \ne \gamma(t)$ if $k \notin H$. The set \{$\gamma_k$\} is parametrized by the space $K/H$, and $s$ is a focal point of $Q$ of multiplicity (at least) dim$(K/H)$. In particular, if $S$ is a singular orbit, dim$(K/H) \ge 1$.
\begin{proof}
Clearly, $\gamma_k(s) = s$ and $\gamma_k(q) \in Kq$. Furthermore, $kq = q$ if and only if $k \in H$. Since $K$ acts by isometries, the geodesics $\gamma_k(t)$ are horizontal.\qedhere
\end{proof}
\end{lem}

In all cases where the orbit space tiles the plane, the orbit space is either an equilateral triangle, a square, or a triangle with a reflection that gives one of these two. Doubling the equilateral triangle produces a rhombus. We shall consider, therefore, only the case of tilings by rhombi where each rhombus contains at most four copies of $Q$. We shall consider rings of rhombi where ring 0 is the rhombus that contains the image of $p$ and ring $n$ consists of all congruent rhombi surrounding ring $n-1$.

Again, we must allow for the contribution of the index of the operator $A$ (sec.~\ref{Ss:MORSE}) associated to the second fundamental form of $Q$. If we set $m$ equal to its smallest possible value, we shall require at most $n-m$ tile crossings to guarantee that all critical geodesics producing cells of dimension less than or equal to $n$ have been obtained. It follows that all horizontal geodesics from $p$ to $Q$ with an index less than or equal to $n$ are contained in rings $0$ through $n-m$. Since there are $(2(n-m)+1)^2$ tiles in rings $0$ through $n-m$ and at most four copies of $Q$ in each tile, it follows that
\begin{center}$\sum_{i=0}^{n} \beta_i \le 4(2(n-m)+1)^2$.\end{center}
 This quadratic bound proves the main theorem for the case of zero curvature and completes the proof of the main theorem for constant curvature.

\subsection{Non-Constant Curvature}
Our approach in the cases which do not admit tilings is the following. We judiciously choose two points $a$ and $b$, typically the most singular orbits in $M/G$, and a simple smooth arc $\gamma$ such that $M/G - |\gamma|$ consists of two components, $C_a$ and $C_b$, containing $a$ and $b$, respectively. Here $|\gamma|$ denotes the support of $\gamma$. Set 
\begin{center}
$A = \pi^{-1}(a)$, $B = \pi^{-1}(b)$, $E = \pi^{-1}(|\gamma|)$

$D(A) = \pi^{-1}(C_a)$, $D(B) = \pi^{-1}(C_b)$, 
\end{center}
where $\pi : M \rightarrow M/G$ is the projection of $M$ onto the orbit space $M/G$. Clearly $E$ is the common boundary of $D(A)$ and $D(B)$. We must choose $\gamma$ so that $E$ is a $C^\infty$-submanifold of $M$. Since the action of $G$ restricted to any of $A$, $B$ and $E$ has at most cohomogeneity one, we know by the Grove-Halperin theorem that each is elliptic. To finish the proof, we need to show that $D(A)$ and $D(B)$ are disk bundles over $A$ and $B$, respectively. Then, by the method of proof of that theorem, it will follow that $M$ is elliptic.

By the Slice Theorem, the normal bundles of $A$ and $B$ can be expressed as equivariant tubular neighborhoods $N_A(\rho_A)$ and $N_B(\rho_B)$. The fibers are open disks of common radii $\rho_A$ and $\rho_B$, respectively. If necessary, we shall reduce the radii of the normal bundles so that their closures have disjoint neighborhoods, and continue to use $\rho_A$ and $\rho_B$ to denote them. To avoid repetition, we will concentrate on the orbit $A$ in what follows. Of course, similar remarks will hold for $B$ as well. We denote the closure of a set $S$ by $\overline{S}$ and the interior by $S^o$ in what follows.

Our immediate task is to extend the bundle $N_A(\rho_A/2)$  to a larger disk bundle $M$ - $\overline{N_B(\rho_B/2)}$  over $A$. We begin by defining a vector field $R$ on $\pi(N_A(\rho_A))$  which is a radial vector field along the geodesics emanating from $a$. Set
\begin{center}$T = M/G$ - $\pi(N_{A}(\rho_A/2))$ $\cup$ $\pi(N_{B}(\rho_B/2))$\end{center}
where $\pi(N_{A}(\rho_A/2))$ and $\pi(N_{B}(\rho_B/2))$ are the disks of radius $\rho_A/2$ and $\rho_B/2$ centered at $a$ and $b$, respectively. Now $T$ is diffeomorphic to either $S^1 \times I$ or to $I \times I$. This diffeomorphism defines a $C^\infty$ family of $C^\infty$ curves $\gamma_s(t)$, parametrized by the first coordinate. Of course any such diffeomorphism takes the boundary of $I$ into the boundary of $M/G$, when $M/G$ has a boundary. Choose the diffeomorphism so that the radial curves on the annulus defined by the disks of radius $\rho_A/2$ and $\rho_A$ agree in some manner with the curves $\gamma_s(t)$, say for $0 \le t \le 1/4$, and do the same for the corresponding annulus around $b$ on the interval $3/4 \le t \le 1$. Define a unit vector field $Y$ on $T$ in such a way that $Y_q$ is the unit tangent vector to the curve $\gamma_s$ through $q$ in the direction of $b$. Choosing a partition of unity $\{\rho_1,\rho_2\}$ subordinate to the open sets $\pi(N_{A}(\rho_A/2))$ and $T^o$, we obtain a vector field $X_p = \rho_1(p)R_p + \rho_2(p)Y_p$ on $M/G$ - $\overline{\pi(N_{B}(\rho_B/2))}$ which is non-singular, except at $a$.  The vector field $X$ has a unique $C^\infty$ horizontal lift $Z$, even when $\partial M/G \ne \varnothing$ since, on the boundary, the lift of the vector is parallel to the orbit stratum. The integral curves of $Z$ fill out the space $M$ - $\overline{N_{B}(\rho_B/2)}$. It follows that  $M$ - $\overline{N_{B}(\rho_B/2)}$ is a disk bundle over $A$. The restriction of this bundle to $\pi^{-1}(C_a)$ is the disk bundle $D(A)$. We obtain the disk bundle $D(B)$ by reversing the direction of the tangent vectors on $T$.
 
To finish the proof, we need only specify the orbits $a$, $b$, and the point set $|\gamma|$ for each case. In case 2, take $a$ to be the exceptional orbit and let $b$ be any other point. For case 3 choose $a$ and $b$ to be the two exceptional orbits. For case 9, take $a$ to be the vertex point and let $b$ be any other point on the boundary. Finally, for case 10, choose $a$ and $b$ to be the two vertices.

It remains to specify the curve $\gamma$. For cases 2 and 3, it suffices to set $\gamma(s) = \gamma_s(1/2)$. For cases 9 and 10, choose the points $\gamma_0(1/2)$ and $\gamma_1(1/2)$ on the boundary of $T$. Let $p$ be a point of the orbit $P = \pi^{-1}(\gamma_0(1/2))$ and consider the normal space $N_p$ to the orbit $P$ at $p$. Then, by the Slice Theorem, there is a disk $D$ centered at 0 in $N_p$ such that $G \times_{G_{p}}D_{p}^{\perp}$ is equivariantly diffeomorphic to a sufficiently small tubular neighborhood of $P$. Since $P$ lies on an orbit stratum of codimension 1 in $M/G$, the image of this bundle is a smooth arc $\alpha_0$ in $M/G$. Since $\alpha_0$ is perpendicular to the orbit stratum it follows that, for sufficiently small $\epsilon_0 > 0$, $\alpha_0(t)$ is transversal to the curves $\gamma_s$ for $t$ in the interval $[0,\epsilon_0)$. Set $\epsilon'_0 = min(\epsilon_0,s/4)$. A similar result holds for the point $\gamma_1(1/2)$, giving an arc $\alpha_1(t)$ which is transversal in $[0,\epsilon'_1)$. Let $\alpha$ be any $C^\infty$ curve, smoothly connecting $\alpha_0(\epsilon'_0/2)$ and $\alpha_1(\epsilon'_1/2)$, crossing each $\gamma_s(t)$ transversally, and lying entirely in the interval $1/4 <  t < 3/4$. The resulting curve satisfies all of our requirements. This completes the proof for these four cases.

\newpage
\bibliographystyle{amsplain}

\end{document}